\documentclass[12pt]{amsart}
\usepackage{amssymb,amsmath,amsfonts,epsfig,latexsym, xypic, tabu}
\usepackage{enumerate}
\usepackage{comment}
\usepackage{tikz-cd}
\usepackage{graphicx}
\graphicspath{ }

\newtheorem{theorem}{Theorem}[section]

\newtheorem{lemma}[theorem]{Lemma}

\newtheorem{conjecture}[theorem]{Conjecture}

\theoremstyle{definition}
\newtheorem{definition}[theorem]{Definition}
\newtheorem{remark}[theorem]{Remark}
\newtheorem{example}[theorem]{Example}

\DeclareMathOperator{\Conf}{Conf}

\DeclareMathOperator{\Link}{Link}

\DeclareMathOperator{\Mat}{Mat}

\DeclareMathOperator{\Supp}{Supp}

\begin{document}

\title[Models for configuration space in a simplicial complex]{Models for configuration space \\ in a simplicial complex}

\author{John D. Wiltshire-Gordon}
\address{\tiny{UW-Madison Department of Mathematics\\
Van Vleck Hall\\
480 Lincoln Drive\\
Madison, WI 53706}}
\email{jwiltshiregordon@gmail.com}

\date{June 20, 2017}

\begin{abstract}
We produce combinatorial models for configuration space in a simplicial complex, and for configurations near a single point (``local configuration space.'')  The model for local configuration space is built out of the poset of poset structures on a finite set.  The model for global configuration space relies on a combinatorial model for a simplicial complex with a deleted subcomplex.  By way of application, we study the nodal curve $y^2 z = x^3 + x^2 z$, obtaining a presentation for its two-strand braid group, a conjectural presentation for its three-strand braid group, and presentations for its two- and three-strand local braid groups near the singular point.
 \end{abstract}



\maketitle

\section{Introduction}
\noindent
If $Z$ is a topological space, define the \textbf{configuration space} of $n$-tuples
$$
\Conf(n, Z) = \{ (z_1, \ldots, z_n) \in Z^n \mbox{ so that $i \neq j \implies z_i \neq z_j$ } \}.
$$
Let $X$ be an abstract simplicial complex, and write $|X|$ for its geometric realization.  This note provides combinatorial models for $\Conf(n, |X|)$ and  $\Conf(n, C^{\circ}|X|)$, where $C^{\circ}|X|$ denotes the open cone on $|X|$.

\begin{definition}
Given a partial ordering of the vertices of $X$ that restricts to a total order on every face $\sigma \subseteq X$, a \textbf{conf matrix}
is a matrix $( v_{ij} )$ of vertices in $X$ so that
\begin{enumerate}
\item each row is weakly monotone increasing in the vertex-ordering,
\item for every row, the vertices appearing within that row form a face of $X$, and
\item no row is repeated. \label{item:distinct}
\end{enumerate}
A conf matrix is called \textbf{minimal} if deleting a column results in duplicate rows, no matter which column is deleted.
\end{definition}
\begin{example} \label{example:triangle}
If $X$ has vertices $\{1, 2, 3\}$, ordered as usual, and facets $\{12, 13, 23\}$, then there are $60$ minimal conf matrices with three rows.  By (\ref{item:distinct}), the symmetric group $S_3$ acts freely on these matrices by permuting the rows.  We list a representative from each orbit:
$$
\scalebox{.66}{\mbox{$
\left[
\begin{array}{c}
 1 \\
 2 \\
 3 \\
\end{array}
\right]\;\left[
\begin{array}{cc}
 1 & 1 \\
 1 & 2 \\
 2 & 2 \\
\end{array}
\right]\;\left[
\begin{array}{cc}
 1 & 1 \\
 1 & 3 \\
 2 & 3 \\
\end{array}
\right]\;\left[
\begin{array}{cc}
 1 & 1 \\
 1 & 3 \\
 3 & 3 \\
\end{array}
\right]\;\left[
\begin{array}{cc}
 1 & 2 \\
 1 & 3 \\
 2 & 2 \\
\end{array}
\right]\;
\left[
\begin{array}{cc}
 1 & 2 \\
 1 & 3 \\
 2 & 3 \\
\end{array}
\right] \;
\left[
\begin{array}{cc}
 1 & 2 \\
 1 & 3 \\
 3 & 3 \\
\end{array}
\right] \;
\left[
\begin{array}{cc}
 1 & 2 \\
 2 & 2 \\
 2 & 3 \\
\end{array}
\right]\;\left[
\begin{array}{cc}
 1 & 3 \\
 2 & 2 \\
 2 & 3 \\
\end{array}
\right]\;\left[
\begin{array}{cc}
 2 & 2 \\
 2 & 3 \\
 3 & 3 \\
\end{array}
\right]
$.}}
$$
\end{example}
\begin{theorem} \label{theorem:global}
Let $C(n, X)$ be the simplicial complex whose vertices are the minimal conf matrices for $X$ with $n$ rows, and where a collection of matrices forms a face if their columns can be assembled into a single conf matrix. There is a homotopy equivalence
$$
|C(n, X) | \simeq \Conf(n, |X|).
$$
\end{theorem}

\begin{example}
The minimal conf matrices listed in Example~\ref{example:triangle} give the vertices of $C(3, X)$.  There are $48$ facets.  Up to $S_3$-symmetry, they correspond to the following conf matrices:
$$
\scalebox{.55}{\mbox{$
\left[
\begin{array}{cccc}
 1 & 1 & 1 & 2 \\
 1 & 1 & 3 & 3 \\
 2 & 3 & 3 & 3 \\
\end{array}
\right]\, \left[
\begin{array}{cccc}
 1 & 1 & 1 & 2 \\
 1 & 2 & 2 & 2 \\
 2 & 2 & 3 & 3 \\
\end{array}
\right]\,\left[
\begin{array}{cccc}
 1 & 1 & 1 & 2 \\
 1 & 3 & 3 & 3 \\
 2 & 2 & 3 & 3 \\
\end{array}
\right]\,\left[
\begin{array}{cccc}
 1 & 1 & 1 & 3 \\
 1 & 1 & 2 & 2 \\
 2 & 3 & 3 & 3 \\
\end{array}
\right]\,
\left[
\begin{array}{cccc}
 1 & 1 & 1 & 3 \\
 1 & 2 & 2 & 2 \\
 2 & 2 & 3 & 3 \\
\end{array}
\right]\,\left[
\begin{array}{cccc}
 1 & 1 & 2 & 2 \\
 1 & 3 & 3 & 3 \\
 2 & 2 & 2 & 3 \\
\end{array}
\right]\,\left[
\begin{array}{cccc}
 1 & 1 & 2 & 2 \\
 2 & 2 & 2 & 3 \\
 2 & 3 & 3 & 3 \\
\end{array}
\right]\,\left[
\begin{array}{cccc}
 1 & 1 & 3 & 3 \\
 1 & 2 & 2 & 2 \\
 2 & 2 & 2 & 3 \\
\end{array}
\right]$.}}
$$
The vertices of the facet corresponding to $M$ are the minimal conf matrices whose columns appear in $M$.
\end{example}

\begin{remark}[Symmetric group action]
The complex $C(n, X)$ carries a natural action of the symmetric group $S_n$ by permuting rows.  Under geometric realization, this action matches the usual $S_n$ action on configuration space $\Conf(n, |X|)$.
\end{remark}

\begin{remark}[Automorphisms of $X$]
The model produced in Theorem \ref{theorem:global} also carries any symmetry present in the simplicial complex $X$, provided this symmetry preserves the vertex-ordering.  
\end{remark}

Configuration space as we have described it so far may be termed ``global'' configuration space since each point is allowed to wander anywhere in $|X|$.  We now introduce \textbf{local configuration space}, modeling configuration space in the open cone near a point $p \in |X|$.  Passing to a subdivision if necessary, we may assume that $p$ is a vertex of $X$.  In this case, a small open neighborhood of $p$ is homeomorphic to the open cone on $\Link_X(p)$.  Consequently, finding a model for local configuration space for all $X$ and all $p \in |X|$ is the same as finding a model for $\Conf(n, C^{\circ} |X|)$ for all $X$, and we take this perspective.

\begin{remark}
The cohomology of local configuration space appears in the stalks of the higher direct images $R^q f_* \underline{\mathbb{Z}}$ where $f$ is the inclusion 
$$
f \colon \Conf(n, |X|) \hookrightarrow |X|^n.
$$
These sheaves feature prominently in Totaro's analysis of the Leray spectral sequence for this inclusion \cite{Totaro96}.  More recently, the master's thesis of L\"utgehetmann \cite[Chapter 4]{lutgehetmann} 
gives a partial description of these sheaves in the case where $X$ is one-dimensional.
\end{remark}

In contrast to the model for $\Conf(n, |X|)$ in Theorem \ref{theorem:global}, the model for $\Conf(n, C^{\circ} |X|)$ to be given in Theorem \ref{theorem:local} does not rely on a vertex ordering.  Nevertheless, we assume for notational convenience that the vertices are $\{1, \ldots, k\}$.

\begin{definition} \label{definition:posetofposets}
The \emph{poset of posets on $n$ labeled elements} $(\mathcal{P}(n), \preceq)$ has
$$
\mathcal{P}(n) = \{ \mbox{ poset structures on the set $\{1, \ldots, n\}$ } \}
$$
where $\preceq $ denotes inclusion of relations.
\end{definition}

\begin{remark}
Study of the poset $\mathcal{P}(n)$ has been undertaken by Serge Bouc in \cite{bouc2013} who computes its M\"obius function and the homotopy types of its intervals. 
\end{remark}

\begin{definition}
If $(S_1, \ldots, S_k) \in \mathcal{P}(n)^k$ is a $k$-tuple of posets, then the \textbf{support} of an element $i \in \{1, \ldots, n \}$ is defined
$$
\sigma_i(S_1, \ldots, S_k) = \{ v \in \{1, \ldots, k\} \mbox{ so that $i$ is not minimal in the poset $S_v$ } \}.
$$
\end{definition}

\begin{definition} \label{definition:local}
Define a poset $\mathcal{S}(n, X) \subseteq \mathcal{P}(n)^k$ whose elements are tuples $(S_1, \ldots, S_k)$ that satisfy
\begin{enumerate}
\item for every $i, j \in \{1, \ldots, n\}$, there is some $S_v$ in which $i$ and $j$ are comparable, and \label{item:related}
\item for every $i \in \{1, \ldots, n\}$, the support $\sigma_i(S_1, \ldots, S_k) \subseteq \{1, \ldots, k\}$ is a face of $X$. \label{item:support}
\end{enumerate}
The ordering comes from imposing $\preceq$ in every coordinate.
\end{definition}

\begin{theorem} \label{theorem:local}
Let $L(n, X)$ be the simplicial complex whose vertices are the minimal elements of the poset $\mathcal{S}(n, X)$ and 
whose faces are those collections that can be dominated by some element.
There is a homotopy equivalence $|L(n, X)| \simeq \Conf(n, C^{\circ}|X|)$ where $C^{\circ}|X|$ denotes the open cone on the realization of $X$.
\end{theorem}

\begin{remark}
Although Theorem \ref{theorem:local} is stated and proved for the open cone $C^{\circ}|X|$, it is possible to find a homotopy equivalence $\Conf(n, C^{\circ}|X|) \simeq \Conf(n, C|X|)$ to configuration space in a closed cone.  Consequently,
$$
|L(n, X)| \simeq \Conf(n, C|X|).
$$
\end{remark}

\begin{example}[Two points in the open cone on a closed interval]
Let $X$ have vertex set $\{1, 2\}$ and a single facet $12$.  Since there are three poset structures on the set $\{1, 2\}$, there are nine elements in the poset $\mathcal{P}(2)^2$.  Eight of these nine have the property that every pair $i, j \in \{1, 2\}$ is related in one of the two posets; the one poset pair without this property is the pair where both posets are discrete.  There are four minimal poset pairs, and these are the vertices of $L(2, X)$:
$$
(\; _{1 \; 2}\;, \; _{1}^2 \;) \;\;\;\; (\; _{1 \; 2}\;, \; _{2}^1 \;) \;\;\;\; ( \; _{1}^2 \;, \; _{1 \; 2}\;)\;\;\;\; ( \; _{2}^1 \;, \; _{1 \; 2}\;).
$$
There are also four facets, all one-dimensional, corresponding to the poset pairs
$$
(\; \; _{1}^2\;, \; _{1}^2 \;) \;\;\;\; (\; \; _{2}^1\;, \; _{1}^2 \;) \;\;\;\; (\; \; _{1}^2\;, \; _{2}^1 \;) \;\;\;\; (\; \; _{2}^1\;, \; _{2}^1 \;),
$$
\begin{figure}
\begin{centering}
\includegraphics{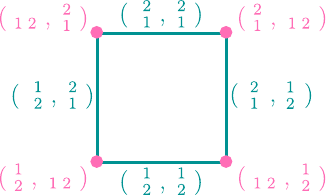}
\end{centering}
\caption{The simplicial complex $L(2, X)$ where $X$ is a $1$-simplex.}
\end{figure}
\noindent
so this computation squares with $\Conf(2, C^{\circ}|X|) \simeq S^1$.
\end{example}

\begin{remark}
The local model is reminiscent of classical work of Fox and Neuwirth describing configuration space in the plane \cite{FoxNeuwirth}.  Roughly speaking, they build their cells out of the three expressions $(x_1 < x_2), \; (x_1 = x_2), \; (x_1 > x_2)$.  Since the middle relation is closed while the other two are open, they obtain a partition of configuration space into locally closed subsets.
Our model does away with the equality relation, lending much greater flexibility.  The price is that we do not partition configuration space; instead, we obtain a model that is homotopy equivalent.
\end{remark}

\begin{remark}
Recent work of Tosteson \cite{Tosteson16} shows that, in many cases, the cohomology of configuration space is representation stable in the sense of \cite{ChurchEllenbergFarb15}.  In particular, the dimension of cohomology over any field is eventually polynomial by \cite{ChurchEllenbergFarbNagpal14}.  It is natural to ask, therefore, if configuration space admits combinatorial models with only polynomially-many cells in each dimension.
\end{remark}

\begin{remark}
Configuration spaces of pairs (i.e., the case $n=2$) is a classical topic.  These spaces are called deleted products, and have applications in embedding theory; see \cite{Hu60}, \cite{Patty61}, \cite{Whitley71}, \cite{Skopenkov02}, for example.
The first results on configuration space in a simplicial complex for $n > 2$ are due to Gal \cite{Gal01} who computes the Euler characteristic of $\Conf(n, |X|)$.  The homotopy groups of certain generalized deleted products of simplicial complexes have been considered in \cite{KS16}.
\end{remark}

\section{Braid groups}
\noindent
The fundamental group $P_n(Z) = \pi_1 \Conf(n, Z)$ is called the \textbf{pure braid group on $n$ strands in $Z$} and sits inside $B_n(Z) = \pi_1 (\Conf(n, Z)/S_n)$, the \textbf{braid group on $n$ strands in $Z$}.  These groups are related by the short exact sequence
$$
1 \longrightarrow P_n(Z) \longrightarrow B_n(Z) \longrightarrow S_n \longrightarrow 1.
$$
The now-common interpretation of these fundamental groups as braid groups goes back to the paper of Fox and Neuwirth \cite{FoxNeuwirth}.

Theorems \ref{theorem:global} and \ref{theorem:local} give presentations for the fundamental groupoids of configuration spaces, which makes algorithmic the computation of braid groups, both $B_n$ and $P_n$.

\begin{remark}
The computation of $B_n$ is already algorithmic by results of An-Park \cite{AP17}, where a description of these groups is obtained in terms of decompositions into simpler spaces.
\end{remark}

If $Z$ is a smooth manifold, then the study of $P_n(Z)$ usually begins with a theorem of Fadell-Neuwirth \cite{FadellNeuwirth62} stating that the forgetful maps $\Conf(n+1, Z) \to \Conf(n, Z)$ are fibrations with fiber $Z - \{ \mbox{$n$ points} \}$.  
However, if $Z$ has singularities, then these forgetful maps are not fibrations, since the homotopy type of $Z - \{ \mbox{$n$ points} \}$ depends on which points are removed. 
Even the case of $n=2$ and $\dim Z = 1$ is interesting and subtle; see \cite{BarnettFarber09}.   More generally, configuration spaces of graphs attract ongoing research interest.  We contribute the following example.

\begin{example} \label{example:k42}
Let $Z$ be the complete bipartite graph $K_{4,2}$, and let $\Sigma_g$ be the smooth surface of genus $g$.  In this and all subsequent examples, we use Sage \cite{sage}---and therefore, indirectly, the group theory program \cite{gap}---to compute braid groups:
\begin{align*}
P_3(Z) &= \pi_1 \, \Sigma_{13} \\
B_3(Z) &= \pi_1 \, \Sigma_{3}.
\end{align*}
By a result of Ghrist \cite{Ghrist01}, the configuration space of any graph  besides the circle and the interval is a $K(\pi, 1)$, and so we actually have $\Conf(3, Z) \simeq \Sigma_{13}$.  
Abrams gave a similar result for two points in $K_5$ and the complete bipartite graph $K_{3,3}$; these spaces are homotopy equivalent to $\Sigma_6$ and $\Sigma_4$ respectively \cite{abrams_thesis}.  That these are the only two-strand braid groups to give a surface group was shown by Ummel \cite{Ummel72}.
\end{example}

\begin{remark}
Our method of computing the quotient $\Conf(n, Z)/S_n$ is to subdivide the model $C(n,X)$ so that pairs of vertices in the same $S_n$-orbit do not share a neighboring vertex.  This endows the set of orbits $\{ \mbox{vertices} \}/ S_n$ with the structure of an abstract simplicial complex whose realization is the topological quotient. 
\end{remark}

The next examples provide what seem to be the first computations of three-strand pure braid groups in two-dimensional singular spaces.

\begin{example}[Braids near the singular point of a nodal curve] \label{example:node}
Let $Z$ be the affine hypersurface $Z = \{(z_1, z_2) \in \mathbb{C}^2\; : \; z_1 z_2 = 0 \}$.  We compute
\begin{align*}
P_2(Z) &= \mathbb{Z} \\
B_2(Z) &= (\mathbb{Z}/2\mathbb{Z}) \ast  (\mathbb{Z}/2\mathbb{Z}) \\
P_3(Z) &= \mathbb{Z} \ast \mathbb{Z} \ast \mathbb{Z} \ast \mathbb{Z} \ast \mathbb{Z} \\
B_3(Z) &= \langle a, b, c, d \mid a^2, b^3, (ab)^2, c^2, d^3, (cd)^2 \rangle.
\end{align*}
Further computing
\begin{align*}
H_2(\Conf(2, Z)\,;\; \mathbb{Q}) &= \mathbb{Q}^2 \\
H_2(\Conf(3, Z)\,; \; \mathbb{Q}) &= \mathbb{Q}^{18},
\end{align*}
we see that neither $\Conf(2, Z)$ nor $\Conf(3, Z)$ is an Eilenberg-MacLane space, since free groups have no homology beyond dimension one.  This is perhaps surprising in light of the result of Ghrist that $\Conf(n, G)$ is aspherical for $G$ a graph, and similarly for $\Conf(n, \mathbb{R}^n)$ by a result of Fadell-Neuwirth \cite{FadellNeuwirth62}.
\end{example}

\begin{example}[Braids in the nodal cubic] \label{example:degenerate_elliptic}
Let $Z$ be the Riemann sphere with two points identified; this is a topological description of the degenerate elliptic curve $y^2 z = x^3 + x^2 z$ advertized in the abstract.  We compute
\begin{align*}
P_2(Z) &= \mathbb{Z} \ast \mathbb{Z} \\
B_2(Z) &= \mathbb{Z} \ast (\mathbb{Z}/2\mathbb{Z}).
\end{align*}
Our computational methods have not been powerful enough to compute $B_3(Z)$ directly.  However, $Z$ is an elementary space in the sense of An-Park, and so $B_3(Z) \simeq \mathbb{Z}^3 \rtimes S_3$ by \cite[Lemma 4.4]{AP17}.
We give a guess for $P_3$ in Conjecture \ref{conjecture:nodal_cubic}.
\end{example}

\begin{example}[Braids in the suspension of $S^1 \sqcup \ast$] \label{example:sphere_plus_edge}
Let $Z$ be the subset of $\mathbb{R}^3$ given by the union of the unit sphere and a line segment connecting its poles:
$$
Z = \{ (x, y, z) \; \mid \; x^2 + y^2 + z^2 = 1 \} \cup \{ (0, 0, t) \; \mid \; -1 \leq t \leq 1 \}.
$$ 
We compute
\begin{align*}
P_2(Z) &= \mathbb{Z} \ast \mathbb{Z} \\
B_2(Z) &= \mathbb{Z} \ast (\mathbb{Z}/2\mathbb{Z}) \\
P_3(Z) &= \mathbb{Z} \ast \mathbb{Z} \ast \mathbb{Z} \\
B_3(Z) &= \langle a, b, c \mid a^2, b^3, (ab)^2, acac^{-1} \rangle.
\end{align*}
We note that this computation of $B_3(Z)$ matches the semidirect product computed by An-Park mentioned in Example \ref{example:degenerate_elliptic}; the elements $a,b$ generate a symmetric group, and the final relation indicates that the transposition $a$ should act trivially on the copy of $\mathbb{Z}$ generated by $c$.

Let $Z'$ be the same space but removing a small open neighborhood of the equator.  Then $P_3(Z') = \mathbb{Z} \ast \mathbb{Z} \ast \mathbb{Z} \ast \mathbb{Z} \ast \mathbb{Z}$, matching a similar group in Example \ref{example:node}.
\end{example}

Although we have been so-far unable to compute the three-strand pure braid group in the nodal cubic from Example \ref{example:degenerate_elliptic}, we make the following conjecture on the strength of Example \ref{example:sphere_plus_edge}.

\begin{conjecture} \label{conjecture:nodal_cubic}
The three-strand pure braid group $P_3$ for the nodal cubic coincides with the three-strand pure braid group for the space from Example~\ref{example:sphere_plus_edge}. 
\end{conjecture}

\section{Acknowledgements}
Thanks to Melody Chan, Fred Cohen, Jordan Ellenberg, Mark Goresky, Greg Malen, Sam Payne, Jos\'e Samper, and Phil Tosteson for valuable conversations, the staff and organizers at ICERM for the stimulating program ``Topology in Motion'' at which this work began, and MathOverflow user Gabriel C. Drummond-Cole for answering a question related to Example~\ref{example:k42} in MO:269017.

This research was supported by the NSF through the Algebra RTG at the University of Wisconsin, DMS-1502553.

\section{Combinatorial deletion of a subcomplex}
In what follows, let $A \subseteq X$ be a subcomplex of the simplicial complex $X$.  Topologically, subtracting $|X| - |A|$ makes perfect sense since $|A| \subseteq |X|$.  However, the combinatorics of such a subtraction cannot be as straightforward as removing the faces of $A$ from $X$, since this completely ruins the simplicial complex structure.  We put forth the following remedy, which seems not to appear in the literature despite its simplicity.
\begin{definition}
The \textbf{simplicial difference} $X \ominus A$ is the simplicial complex whose vertices are the faces of $X$ that are minimal nonfaces of $A$ and where a collection of such minimal nonfaces $\{ \sigma_0, \ldots, \sigma_d \}$ is a face of $X \ominus A$ whenever the union $\cup \sigma_i$ is a face of $X$.
\end{definition}

\begin{lemma} \label{lemma:delete}
If $X$ is an abstract simplicial complex and $A \subseteq X$ is a subcomplex, then there is a homotopy equivalence $|X \ominus A| \simeq |X| - |A|$.
\end{lemma}
\begin{proof}
Suppose the vertices of $X$ are $\{1, \ldots, k\}$.  One standard definition of the geometric realization of $X$ is
\begin{multline*}
|X| = \\
\{  (\alpha_1, \ldots, \alpha_k) \in \mathbb{R}^k \mbox{ so that $\alpha_i \geq 0$, $\sum_i \alpha_i = 1$, and $\Supp(\alpha)$ is a face of $X$} \},
\end{multline*}
where $\Supp(\alpha) = \{\mbox{ $i$ so that $\alpha_i > 0$ } \}$.  If $\sigma \subset X$ is a nonempty face of $X$, then define an open subset of $|X|$
$$
U_X(\sigma) = \{ (\alpha_1, \ldots, \alpha_k) \in |X| \mbox{ so that $\alpha_i >0$ for $i \in \sigma$ } \}.
$$
The set $U_X(\sigma)$ is nonempty because we may take $\alpha_i = 1/(d + 1)$ for $i \in \sigma$ and $\alpha_i = 0$ when $i \not \in \sigma$.  Call this tuple $b_{\sigma}$ since it is the barycenter of the realization of $\sigma$.  The property of $b_{\sigma}$ that we will need is that $\Supp(b_{\sigma}) \subseteq \Supp(u)$ for any other $u \in U_X(\sigma)$.
We shall see that $U_X(\sigma)$ is contractible---indeed, we now show that it is star-shaped with star point $b_{\sigma}$.

We must show that, given any other point $u \in U_X(\sigma)$ and $\gamma \in [0, 1]$, the linear combination $\gamma u + (1-\gamma) b_{\sigma}$ lies within $U_X(\sigma)$.  This amounts to checking that $\Supp(\gamma u + (1-\gamma) b_{\sigma})$ is a face of $X$.  However, it is immediate that
$$
\Supp(\gamma u + (1-\gamma) b_{\sigma}) = \Supp(u)
$$
for any $\gamma > 0$ since $\Supp(b_{\sigma}) \subseteq \Supp(u)$; and $\Supp(u)$ is certainly a face of $X$ since $u \in |X|$.

We now use some of these open sets $U_X(\sigma)$ to build a good cover of $|X|-|A|$; the result will then follow by Borsuk's nerve lemma.  Note that a point $\alpha \in |X|$ lies in $|X|-|A|$ exactly when $\Supp(\alpha)$ is not a face of $A$.  Since any such nonface must contain a minimal nonface, we see that the open sets $U_X(\sigma)$ cover $X$ as $\sigma$ ranges over the vertices of $X \ominus A$.  To see that this cover is good, we must check that every intersection of these opens is either empty or contractible.  Such an intersection
$$
U_X(\sigma_1) \cap \cdots \cap U_X(\sigma_l)
$$
is empty when the union $\cup \sigma_i$ is not a face of $X$, and otherwise equals $U_X(\cup \sigma_i)$.  Since this last set is already known to be contractible, we are done since the resulting combinatorics matches the definition of $X \ominus A$.
\end{proof}

\begin{remark}
A result of Munkres \cite[Lemma 70.1]{Mun84} gives a different combinatorial model for $|X| - |A|$ in the event that $A$ is a full subcomplex of $X$.  (Recall that $A \subseteq X$ is a full subcomplex if every face of $X$ whose vertices are in $A$ is also a face of $A$.)  The model is simple: take the full subcomplex generated by vertices not lying in $A$.

Any subcomplex can be made full by sufficient subdivision.  However, the resulting model is distinct from the the model we have described, and usually much larger.  For example, let $X=\{123\}$ be a 2-simplex, and let $A=\{12,23,13\}$ be its boundary.  The simplicial difference $X \ominus A$ has one vertex $\{(123)\}$, corresponding to the minimal nonface of $A$, and no edges.  We now consider the model of Munkres.  The most efficient subdivision for which $A$ is full is the barycentric subdivision of $X$, which has seven vertices.  Subtracting the three vertices of $A$, we obtain a model for $|X|-|A|$ with four vertices and three edges.
\end{remark}

\section{Global configuration space}
In this section, we provide background and a proof for Theorem \ref{theorem:global}.  First, we will build a standard triangulation of the product $|X|^n$ so that the ``fat diagonal'' of illegal configurations is a subcomplex.  Then, we will apply Lemma \ref{lemma:delete} to remove the fat diagonal.

Products are more convenient in the context of simplicial sets instead of simplicial complexes because the product of two simplicial sets comes with a natural triangulation.
The extra structure required to upgrade a simplicial complex to a simplicial set is a partial ordering on the vertices that restricts to a total order on every face.

\begin{definition}
If $X$ is a simplicial complex with vertex set $\{1, \ldots, k\}$, and if the vertices carry a partial order that restricts to a total order on every face, then define the \textbf{simplicial set associated to $X$} which is a functor
$$
X_{\bullet} \colon \Delta^{op} \to \mathbf{Set}
$$
so that $X_n$ is the set of weakly increasing chains of vertices $v_0 \leq \ldots \leq v_n$ with the property that $\{v_0, \ldots, v_n \}$ is a face of $X$.  Such a chain pulls back along a weakly monotone function $f \colon \{0, \ldots, m\} \to \{0, \ldots, n\}$ using the rule
$$
v_0 \leq \ldots \leq v_n \;\;\;\;\; \mapsto \;\;\;\;\; v_{f(0)} \leq \ldots \leq v_{f(m)}.
$$
\end{definition}

We recall several basic facts about simplicial sets; see \cite[p. 538]{Hatcher}, for example.  Writing $\| - \| \colon \mathbf{sSet} \to \mathbf{Top}$ for the realization functor on simplicial sets, we have $|X| \cong \| X_{\bullet} \|$.  Also, as we have mentioned, we have a homeomorphism $\|S_{\bullet} \times T_{\bullet} \| \cong \|S_{\bullet} \| \times \|T_{\bullet} \|$ for any pair of simplicial sets $S_{\bullet}, T_{\bullet}$.  Finally, we note that the degreewise diagonal map $X_{\bullet} \hookrightarrow X_{\bullet} \times X_{\bullet}$ and projection maps $X_{\bullet} \times X_{\bullet} \to X_{\bullet}$ realize to their topological counterparts $\| X_{\bullet} \| \hookrightarrow \| X_{\bullet} \| \times \| X_{\bullet} \|$ and $\|X_{\bullet}\| \times \|X_{\bullet}\| \to \|X_{\bullet}\|$.
\begin{lemma}[Triangulation of a product {\cite[\S 3.B]{Hatcher} or \cite[p. 68]{EilenbergSteenrod1952}}] \label{lemma:triangulation}
Suppose that $X$ and $Y$ are simplicial complexes, and that each is equipped with a partial order on its vertices that restricts to a total order on each of its faces.  Write $X \times Y$ for the simplicial complex whose vertices are pairs of vertices $(x, y)$ with $x \in X, y \in Y$ and where the faces are of the form $\{ \, (x_1, y_1), \ldots, (x_d, y_d) \, \}$ with $x_1 \leq \cdots \leq x_d$ and $y_1 \leq \cdots \leq y_d$.  Then
$$
|X \times Y| \cong |X| \times |Y|.
$$
\end{lemma}
\begin{proof}
In brief, we have $|X \times Y| \cong \| (X \times Y)_{\bullet} \| \cong \| X_{\bullet} \times Y_{\bullet} \| \cong \| X_{\bullet} \| \times \| Y_{\bullet} \| \cong |X| \times |Y|$.
\end{proof}

\begin{proof}[Proof of Theorem \ref{theorem:global}]
Triangulating $|X|^n$ in the standard way described in Lemma \ref{lemma:triangulation}, we obtain a simplicial complex $X^n$ with a homeomorphism $|X^n| \cong |X|^n$.  In detail, the vertices of $X^n$ are $n$-tuples of vertices of $X$, and a collection of $n$-tuples forms a face if they can be arranged to be weakly increasing in every coordinate.  Writing each $n$-tuple as a column vector, the faces are matrices with $n$ rows where each row is weakly increasing.

The fat diagonal $F \subseteq X^n$ is the subcomplex whose faces are matrices that have a repeated row, so a matrix is a face of $F$ if it is almost a conf matrix but fails condition (\ref{item:distinct}).  By Lemma \ref{lemma:delete} we have a homotopy equivalence 
$$
|X^n \ominus F| \simeq |X^n| - |F| \cong |X|^n - |F| = \Conf(n, |X|).
$$
Not by coincidence, the definition of $C(n, X)$ matches that of $X^n \ominus F$, and we are done.
\end{proof}

\section{Local configuration space}
\noindent
Define the \textbf{cone realization} of a simplicial complex with vertices $\{1, \ldots, k\}$
\begin{multline*}
|X|_{cone} = \\
\{ \mbox{ $(\alpha_1, \ldots, \alpha_k) \in \mathbb{R}^k$ so that $\alpha_i \geq 0$ for all $i$ and $\Supp(\alpha)$ is a face of $X$ } \},
\end{multline*}
where $\Supp(\alpha) = \{ \mbox{ $i$ so that $\alpha_i > 0$ } \}$ records the positions of the positive entries of $\alpha$.  The usual realization $|X| \subset |X|_{cone}$ is the intersection of the cone realization with the hyperplane $\sum_i \alpha_i = 1$.   We have a homeomorphism
$$
|X|_{cone} \cong C^{\circ} |X|
$$
to the open cone on $|X|$.
The $n$-fold product $(|X|_{cone})^n$ sits inside the matrix space $\Mat_{\mathbb{R}}(n \times k)$ where each row gives a $k$-tuple that lives in $|X|_{cone}$:
\begin{multline*}
(|X|_{cone})^n \simeq \\
\left\{ \parbox{28em}{ real matrices $M \in \Mat_{\mathbb{R}}(n \times k)$ so that $M_{ri} \geq 0$ and for each $r \in \{1, \ldots n\}$, the $r^{\tiny \mbox{th}}$ row $(M_{r1}, M_{r2}, \ldots, M_{rk})$ lies in $|X|_{cone}$ } \right\}.
\end{multline*}
Such a matrix $M$ gives a point of $\Conf_n(|X|_{cone})$ when its rows are distinct:
\begin{multline*}
\Conf_n(|X|_{cone}) \simeq \\
\left\{ \parbox{29em}{ real matrices $M \in (|X|_{cone})^n \subseteq \Mat_{\mathbb{R}}(n \times k)$ having distinct rows  } \right\}.
\end{multline*}
For any $i \in \{1, \ldots, k\}$ and partial order $P$ on the set $\{1, \ldots, n\}$, we define an open subset of $(|X|_{cone})^n$ by asking that column $i$ of the matrix $M$ obey the strict inequalities present in the partial order $P$:
$$
U_P^i = \left\{ \parbox{21em}{ real matrices $M \in (|X|_{cone})^n$ so that for any $a, b \in \{1, \ldots, n\}$ we have $a <_P b \implies M_{ai} < M_{bi}$ } \right\}.
$$
If no two points $a, b \in \{1, \ldots, n\}$ are comparable in $P$, then the open set $U_P^i$ imposes no condition.  At the other extreme, if $P$ is a total order, then the open set $U_P^i$ consists of those matrices for which the real numbers in column $i$ are sorted according to $P$.

\begin{lemma} \label{lemma:contractible}
If $S = (S_1, \ldots, S_k)$ is a $k$-tuple of partial orders on the set $\{1, \ldots, n\}$, then the intersection
$$
U_S = U_{S_1}^1 \cap \cdots \cap U_{S_k}^k
$$
is either contractible or empty.
\end{lemma}
\begin{proof}
The intersection in question may be written
\begin{equation*} 
U_S = \left\{ \parbox{26em}{ real matrices $M \in (|X|_{cone})^n$ so that for any $a, b \in \{1, \ldots, n\}$ and $i \in \{1, \ldots, k \}$ we have $a <_{S_i} b \implies M_{ai} < M_{bi}$ } \right\}.
\end{equation*}
We show that this intersection is star-shaped by finding a star-point matrix whose entries are zero ``wherever possible.''  Write $\varepsilon_i \subseteq \{1, \ldots, n\}$ for the minimal elements of poset $S_i$, and $\eta_i \subseteq \{1, \ldots, n\}$ for the non-minimal elements, so that $\{1, \ldots, n\} = \varepsilon_i \sqcup \eta_i$.  For any non-minimal $b \in \eta_i$ and any $M \in U_S$, it must be that $0 < M_{bi}$.  Indeed, since $b$ is not minimal, there exists $a$ with $a <_{S_i} b$, and so $0 \leq M_{ai} < M_{bi}$.  In other words, every matrix $M \in U_S$ has certain entries that must be nonzero.  We now show that there exists a matrix $M^{\star} \in U_S$ that is nonzero at exactly these times, meaning $M_{ai}^{\star} = 0$ for all $a \in \varepsilon_i$.  In order to fill in the other entries of $M^{\star}$, we must choose positive values for $M_{bi}^{\star}$ with $b \in \eta_i$.  Restrict the partial order on $S_i$ to the subset $\eta_i$.  Extend this restricted order to a total order arbitrarily, and embed the resulting total order in the positive reals $\mathbb{R}_{>0}$.  This embedding extends to a function $\varphi_i \colon S_i \to \mathbb{R}_{\geq 0}$ so that $\varphi_i(a) = 0$ for $a \in \varepsilon_i$ and $\varphi_i(b) > 0$ for $b \in \eta_i$, using minimality.  Setting $M_{ai}^{\star} = \varphi_i(a)$, we have produced a matrix of minimum support: its entries are zero everywhere that a zero is possible.

In order to see that $M^{\star}$ is a star point, choose some other $M \in U_S$.  We must show that every convex combination $M^{\delta} = \delta M^{\star} + (1-\delta) M$ remains in $U_S$, where $0 \leq \delta \leq 1$.  Since every positive entry of $M^{\star}$ is also positive in $M$, the support of $M^{\delta}$ matches the support of $M$.  As a result, the rows of $M^{\delta}$ are still elements of $|X|_{cone}$.  The other conditions defining $U_S$ are of the form $M_{ai} < M_{bi}$, and these conditions are convex, which means they are automatically satisfied by $M^{\delta}$.  In detail, suppose $a <_{S_i} b$.  Then, since both $M^{\star}$ and $M$ are in $U_S$, we have $M_{ai}^{\star} < M_{bi}^{\star}$ and $M_{ai} < M_{bi}$.  It immediately follows that $\delta M_{ai}^{\star} + (1-\delta)M_{ai} < \delta M_{bi}^{\star} + (1-\delta)M_{bi}$, and so $M^{\delta}_{ai} < M^{\delta}_{bi}$ as required.
\end{proof}

\begin{proof}[Proof of Theorem \ref{theorem:local}]
Lemma \ref{lemma:contractible} supplies open subsets of $(|X|_{cone})^n$ that we will use to make a good cover of $\Conf(n, |X|_{cone})$.  We will show that the combinatorics of this cover match the combinatorics defining $L(n, X)$, concluding the proof by the nerve lemma.

For every tuple of posets $S = (S_1, \ldots, S_k) \in \mathcal{S}(n, X)$, we have defined an open set $U_S$.  Using condition (\ref{item:support}) in Definition~\ref{definition:local}, these opens are nonempty.  They are therefore contractible by Lemma~\ref{lemma:contractible}.  Because of (\ref{item:related}), each of these opens is a subset of configuration space $U_S \subseteq \Conf(n, |X|_{cone})$.

Given any configuration of distinct points in $|X|_{cone}$, we may record the partial order on $\{1, \ldots, n\}$ induced by each coordinate.  Note that the ordering may not be a total order because two points may tie in any given coordinate.  However, in order for the points to be distinct, any two must differ in some coordinate.  This argument shows that the sets $U_S$ with $S \in \mathcal{S}(n, X)$ form an open cover.

Next, we note that any intersection of open subsets drawn from the cover $\{ \, U_S \, \}_{S \in \mathcal{S}(n, X)}$ is either empty or else still a set from the cover.  The reason is that a collection of partial orders on $\{1, \ldots, n \}$, when simultaneously imposed and closed under transitivity, either leads to a contradiction or else to a partial order extending all of those in the collection in a minimal way.

The vertices of $L(n, X)$ consist of the minimal elements of $\mathcal{S}$.  The corresponding open sets form an open cover since any other $U_S$ contained in one with $S$ minimal.  An intersection of such opens is non-empty exactly when there is some $S \in \mathcal{S}$ so that $U_S$ is the intersection.  We have already seen that $U_S$ is contractible, and so $L(n, X)$ is the simplicial complex that records the combinatorics of the good cover provided by the minimal elements of $\mathcal{S}$.  The result then follows from Borsuk's nerve lemma.
\end{proof}

\bibliographystyle{amsalpha}
\bibliography{math}

\end{document}